\documentclass[preprint,10pt]{elsarticle}

\usepackage{color}

\usepackage{amsmath}

\usepackage{epsfig,graphicx,psfrag,mathrsfs}
\usepackage{amsmath,amssymb,dsfont,upgreek,multirow}
\usepackage{hyperref,color,url}

\usepackage{algorithm}
\usepackage{algorithmic}

\newtheorem{Problem}{Problem}[section]

\newtheorem{theorem}{Theorem}[section]
\newtheorem{lemma}{Lemma}[section]

\newtheorem{example}{Example}[section]

\newtheorem{corollary}{Corollary}[section]
\newproof{proof}{Proof}
\newproof{pot}{Proof of Theorem \ref{thm2}}

\newcommand{\bb}{\begin{bmatrix}}
\newcommand{\eb}{\end{bmatrix}}

\begin{document}

\date{}
\begin{frontmatter}
\title{
A fast numerical algorithm for constructing nonnegative
matrices with prescribed real eigenvalues
\tnoteref{t1}}%,t2}}

\author{Matthew M. Lin \corref{cor1}\fnref{fn2}}
\ead{mhlin@ccu.edu.tw}
\address{Department of Mathematics, National Chung Cheng University, Chia-Yi 621, Taiwan.}
\cortext[cor1]{Corresponding author}

\fntext[fn2]{The author was supported by the National Science Council of Taiwan under grant
NSC101-2115-M-194-007-MY3.}
%\fntext[fn3]{Yet another author footnote. Indeed, you can have
%any number of author footnotes.}

\date{ }

\begin{abstract}
The study of solving the inverse eigenvalue problem for nonnegative matrices has been around for decades. It is clear that an inverse eigenvalue problem is trivial if the desirable matrix is not restricted to a certain structure. 
Provided with the real spectrum, this paper presents a numerical procedure, based on the induction principle, to solve two kinds of inverse eigenvalue problems, one for nonnegative matrices and another for symmetric nonnegative matrices. As an immediate application, our approach can offer not only the sufficient condition for solving inverse eigenvalue problems for nonnegative or symmetric nonnegative matrices, but also a quick numerical way to solve inverse eigenvalue problem for stochastic matrices. Numerical examples are presented for 
problems of relatively larger size.

\end{abstract}


\end{frontmatter}

\section{Introduction}

A real $n\times n$ matrix is said to be nonnegative if each of its entries is nonnegative. Considerable research efforts have been directed towards the properties of the eigeninformation of  nonnegative matrices, especially the following nonnegative inverse eigenvalue problem (NIEP).
%
%. Among the large number of established results, one well-known necessary, but insufficient, condition is the Perron-Frobenius theory. This theory asserts first that 
%for a real square matrix with positive entries, there exists a unique largest real eigenvalue with the corresponding eigenvector having strictly positive components. It also asserts a similar statement for some classes of nonnegative matrices~\cite{Gantmacher1959}. 
%% Inverse eigenvalue problems
%
%In contrast to this well-developed theory, Kolmogorov in~\cite{Kolmogorov1937} showed that every complex number is an eigenvalue of some nonnegative matrix.%~\cite{Kolmogorov1937}. 
%%In,  
%Later, Suleimanova~\cite{Sulei1949} generalized Kolmogorov's result to the study of the well-known problem --
%%the \emph{nonnegative inverse eigenvalue problem} (NIEP) and 
%the \emph{real nonnegative inverse eigenvalue problem} (RNIEP).
%%\begin{Problem}[NIEP]
%%Let $\sigma = \{\lambda_1,\ldots,\lambda_n\}$ be a set of $n$ complex numbers. Find necessary and sufficient conditions for $\sigma$ to be the set of eigenvalues of some nonnegative $n\times n$ matrix.
%%\end{Problem}
%

\begin{Problem}[NIEP]
Let $\sigma = \{\lambda_1,\ldots,\lambda_n\}$ be a set of $n$ complex numbers. Find a nonnegative $n\times n$ matrix with eigenvalues $\sigma$ (if such a matrix exists).
%necessary and sufficient conditions for $\sigma$ to be the set of eigenvalues of some nonnegative $n\times n$ matrix.
\end{Problem}

It is easy to see that the solution of the NIEP may not be unique, once it  exists, since  there are $n$ given numbers with respect to $n^2$ unknown variables, i.e., an $n\times n$ matrix.  
More generally, let $\sigma 
= \{\lambda_1,\ldots,\lambda_n\}$ be a set of 
eigenvalues of an $n\times n$ matrix $A$ and let the $k$th moment $s_k$ of $\sigma$ be defined by 
%Let $s_k$ be a sequence of numbers given by 
\begin{equation}\label{moments}
s_k = \sum_{i=1}^n \lambda_i^k = \rm{trace}(A^k),\quad k=1,2,\ldots,
\end{equation} 
It follows that if $\sigma$ is a set of eigenvalues of a nonnegative matrix $A$, then the moments of the nonnegative matrix are always nonnegative,i.e.,
\begin{equation}\label{Loewy1}
s_k \geq 0, \quad k = 1,2,\ldots.
\end{equation} 
Based on the notion given in~\eqref{moments}, the following necessary condition provides the most broad-based necessary condition in the solvability of a nonnegative inverse eigenvalue problem and can be shown by simply applying the H\"{o}lder inequality~\cite{Loewy1978}.
%Here the numbers $s_k$ are the so-called \emph{moments} of a given matrix with eigenvalues
%$\{\lambda_1,\ldots,\lambda_n\}$.
%In~\cite{Loewy1978}, Loewy and London provides the most broad-based necessary condition in the solvability of the RNIEP
%by simply applying the H\"{o}lder inequality.

\begin{theorem}
Suppose $\sigma = \{\lambda_1,\ldots,\lambda_n\}$ be a set of eigenvalues of an $n\times n$ nonnegative matrix. Then the inequalities
\begin{equation}\label{Loewy}
s_k^m\leq n^{m-1} s_{km}
\end{equation}
are satisfied for all $k, m = 1,2,\ldots.$
 
\end{theorem}

It has been shown in~\cite{Loewy1978} that 
inequalities~\eqref{Loewy1} and~\eqref{Loewy} are the necessary and sufficient conditions for  
$\sigma = \{\lambda_1,\ldots,\lambda_n\}$
with $n\leq 3$
 to be a set of eigenvalues of some nonnegative matrix. However, for $n\geq 4$,~\eqref{Loewy1} and~\eqref{Loewy} are not sufficient, and the problem is still open. If $\sigma$ is further restricted to be real, i.e., 
the NIEP with real eigenvalues (RNIEP), then conditions~\eqref{Loewy1} and~\eqref{Loewy} are still necessary and sufficient for solving RNIEP with $n=4$~\cite{Loewy1978}. In fact, the RNIEP is still open for $n\geq 5$. Truly, there are various necessary or sufficient conditions for a list $\sigma$ to be realizable as the eigenvalues of a nonnegative matrix; however, in general, the necessary conditions are unusually too general and the sufficient conditions are too specific with nonconstructive proofs~\cite[Section 6]{Chu02}. 
%
%
%On the other hand,  if consideration of nonnegative matrices is limited to positive matrices, i.e., matrices with positive entries, the most general result to fully characterize the eigenvalues is as follows, as per Boyle and Handelmann~\cite{Boyle1991}[p.313].
%\begin{theorem}\label{Boyle}
%The set $\{\lambda_1,\ldots,\lambda_n\}\subset \mathbb{C}$ with $\lambda_1 = \max_{1\leq n} |\lambda_i|$ is the nonzero spectrum of a positive matrix of size $m\geq n$ if and only if the following conditions are satisfied:
%\begin{enumerate}
%\item $\lambda_1 >|\lambda_i|$ for all $i>1$,
%\item $s_k >0$ for all $k=1,2,\ldots,$ and 
%\item All coefficients of the polynomial $\Pi_{i=1}^n(t-\lambda_i)$ are real in $t$.
%\end{enumerate} 
%%
%\end{theorem}
%%
%Although Theorem~\ref{Boyle} provides the necessary and sufficient conditions for the solvability of the inverse eigenvalue problem of positive matrices, these conditions are too complicated for most applications due to the examination of $s_k> 0$ for all $k=1,2,\ldots$. Naturally, the next challenge is to find a way of constructing nonnegative matrices from a given list of eigenvalues. This question has been widely investigated in the literature, but its answer is still open~\cite{Loewy1978}. 
%
%Here, we briefly recall some sufficient conditions that a given real numbers are the spectrum of an nonnegative matrix.
One sufficient condition that is constructive for a list of $n$ real numbers to be the spectrum of a nonnegative matrix is given by Suleimanova~\cite{Sulei1949}.  
%In~\cite{Sulei1949}, Suleimanova 
%first came up with the sufficient condition for a list of $n$ real numbers to be the spectrum of a nonnegative matrix by means of a geometrical approach.
\begin{theorem}\label{Sulei}
Suppose $\sigma =\{\lambda_k\}_{k=1}^n \subset \mathbb{R}$, $\lambda_1+\lambda_2+\ldots+\lambda_n\geq 0$ and  $\lambda_{i} < 0$ for $i = 2,\ldots, n$. Then there exists a nonnegative $n\times n$ matrix with spectrum $\sigma$.
\end{theorem}

Indeed, Suleimanova's result can also be limited to the case of symmetric matrices and a simple proof for the case of symmetric matrices is given in~\cite{Fiedler1974}[Theorem 2.4]. In this paper, a weaker condition than Suleimanova's result is provided for solving RNIEP. We then apply this weaker condition for constructing a nonnegative matrix associated with the given real eigenvalues. 
There are many sufficient conditions for solving RNIEP in the literature~\cite{Perfect55,Kellogg1971,Salzmann1972,Fiedler1974,Berman1994,Borobia1995,Soto2003b,Egleston2004,Soto2007a,Soto2011} and the references contained therein. 
Instead of comparing our condition with other known results, we present here a numerical approach, based on an improvement to Suleimanova's condition~\cite{Sulei1949}, to solve RNIEP of larger size. 

In addition to RNIEP, we also discuss a related problem, called the \emph{symmetric nonnegative inverse eigenvalue problem} (SNIEP), proposed by Fiedler~\cite{Fiedler1974}.
\begin{Problem}[SNIEP]
Let $\sigma = \{\lambda_1,\ldots,\lambda_n\}$ be a set of $n$ real numbers. 
Find a symmetric nonnegative $n\times n$ matrix with eigenvalues $\sigma$ (if such a matrix exists).
%Find necessary and sufficient conditions for $\sigma$ to be the set of eigenvalues of some symmetric nonnegative $n\times n$ matrix.
\end{Problem}
Again, this problem is still open for $n\geq 5$~\cite{Loewy1978,Meehan1998,Torre2007}.

So far as we know, applicable numerical methods for solving SNIEPs have thus far been proposed only twice~\cite{Chu1998,Orsi2006}. 
In~\cite{Chu1998}, the 
SNIEP is formulated as the following  constrained optimization problem
\begin{equation}
\min_{Q^\top Q = I, R=R^\top}\frac{1}{2}\|Q^\top \Lambda Q- R\circ R\|.
\end{equation}
Here, $\Lambda$ is a diagonal matrix with the desired spectrum and $\circ$ represents the Hadamard product. The idea  is to parameterize any symmetric
matrix with the desired spectrum equal to $\Lambda$
by $X=Q^\top\Lambda Q$ and to parameterize 
any symmetric nonnegative matrix $Y$ by $Y = R\circ R$ for some symmetric matrix $R$. Later, Orsi~\cite{Orsi2006} utilizes alternating projection ideas 
for the SNIEP. This projection consists of two particular 
sets. One is the set of all real symmetric matrices with the desired spectrum. The other is the set of symmetric nonnegative matrices.  It should be noted that above both methods are proposed to approximate a nonnegative matrix with the desired spectrum. 

Instead of obtaining an  approximate result, our work constructs a symmetric nonnegative matrix based on a sequence of  $2\times 2$ matrices as a building block. This approach is guaranteed to construct a nonnegative matrix of size $n$ after $n-1$ iterations. It should be noted that there are many other inverse eigenvalue problems involving matrices with a particular structure
and a particular desired spectrum. 
For more on other inverse problems, see the papers~\cite{Chu1991,Chu1998b,Chu1998,Chu02,Chu2005} and the book~\cite{Chu05}. In this paper, we describe a numerical procedure for solving RNIEP and SNIEP, which, while admittedly quite crude, suggests the  possibility of solving many structured inverse eigenvalue problems and is currently under investigation.
%
%Our next work is to handle inverse eigenvalue problems with respect to 
%
% It is true that the methods stated in~\cite{Chu1998,Orsi2006} can be applied to inverse eigenvalue problems  

%
%%
%%
%%
%%
%
%Chu's Results
%\cite{Chu1991,Chu1998b,Chu1998,Chu02,Chu2005,Chu05}-See Orsi discussion.
%
%%
%%
%%
%%

%The structure of this paper is organized as follows.
This paper is organized as follows.
We begin Section 2 with a discussion of the condition of two eigenvalues to be a spectrum of a $2\times 2$ nonnegative matrix and construct this $2\times 2$ matrix explicitly, given its eigenvalues. 
This $2\times 2$ construction then serves as a fundamental tool in the construction of an $n\times n$ 
nonnegative matrix. In Section 3, we briefly review 
Nazari and Sherafat's result~\cite{Nazari2012}
in combining two nonnegative matrices with the desired spectrum. We point out, in particular, how to apply this result to a general $n\times n$ matrix by the splitting of this given matrix. 
%Also, to kwon the way of splitting is a crucial step in our algorithm. 
In Section 4, we discuss 
how the $2\times 2$ construction can be applied to 
the inverse eigenvalue problem for stochastic matrices. Concluding remarks are given in Section 5.

%
%
%
%Then, we  want to demonstrate how to construct $2\times 2$ case in terms of elementary algebra. We then use this $2 \times 2$ case to come up with the idea of generating the general matrix.   

\section{The $2 \times 2$ building block}
In this section, we describe how a nonnegative matrix $A$ can be constructed. Specifically, we want to determine a $2\times 2$ nonnegative matrix $A$ with $\sigma(A) = \{\lambda_1, \lambda_2\}$. This $2\times 2$ construction will become a building block in our recursive algorithm.  
%Numerical experiments suggest that approximately $O(n)$ flops are required to construct an $n\times n$ matrix. 
Note that for the existence of a nonnegative matrix 
\[
A =\left[\begin{array}{cc} a & b \\c & d\end{array}\right], 
%\neq \mathbf{0}%\in\mathbb{R}^{2\times 2}
\]
with eigenvalues $\{\lambda_1, \lambda_2\}$, 
it is true that 
\begin{subequations} \label{genrule}
\begin{eqnarray}
a + d &=& \lambda_1 + \lambda_2 \geq 0,\label{eq1}\\
ad-bc &=& \lambda_1 \lambda_2. \label{eq2}
\end{eqnarray}
\end{subequations}
Since $b$ and $c$ are nonnegative, it follows directly from~\eqref{eq1} and~\eqref{eq2} that 
\begin{eqnarray}
bc &=& a (\lambda_1+\lambda_2-a) -\lambda_1\lambda_2  \nonumber\\
&=& -(a - \frac{\lambda_1+\lambda_2}{2})^2 
+ \frac{(\lambda_1-\lambda_2)^2}{4} \geq 0\label{eq:bc}.
\end{eqnarray}
This implies that $\lambda_1 \geq a \geq \lambda_2$. 
If $\lambda_2 < 0$, then the entry $a$ is further limited to $\lambda_1+\lambda_2\geq a \geq 0$. Putting together the above results, the entries of nonnegative matrices with the set of eigenvalues $\{\lambda_1, \lambda_2\}$ can be completely characterized as follows.
\begin{lemma}\label{LemNon}
$\{\lambda_1, \lambda_2\}$ are eigenvalues of a $2\times 2$ nonnegative matrix
$A = \left[\begin{array}{cc}a & b \\c & d\end{array}\right ]$
if and only if~\eqref{genrule} and the following  conditions,
\begin{align}\label{cond}
\left.\begin{array}{rl}
\lambda_1 \geq a \geq \lambda_2,& 
\mbox{if }  \lambda_2 \geq 0, \\
\lambda_1+\lambda_2 \geq a \geq  0, &
\mbox{if }  \lambda_2 < 0,
%\\b\geq 0, &c\geq 0,
\end{array}\right.
\end{align}
are satisfied. 
\end{lemma}
\begin{proof}
It follows from~\eqref{genrule} and~\eqref{eq:bc} that we need only prove that if~\eqref{genrule} and~\eqref{cond} hold, then $A$ is a nonnegative matrix with the desired spectrum $\{\lambda_1, \lambda_2\}$. Suppose~\eqref{cond} holds. Then~\eqref{eq1} implies that 
\begin{align*}
\left.\begin{array}{rl}\lambda_1 \geq d \geq \lambda_2,& 
\mbox{if }  \lambda_2 \geq 0, \\
\lambda_1+\lambda_2 \geq d \geq  0, &
\mbox{if }  \lambda_2 < 0, 
\end{array}\right.
\end{align*}
 and thus $A$ is an nonnegative matrix.

\end{proof}

From Lemma~\ref{LemNon}, it is straightforward to see that 
the matrix 
\begin{equation}\label{NonSym}
A = \left\{\begin{array}{cc}
\left[\begin{array}{cc}\lambda_2 & 0\\0 & \lambda_1\end{array}\right], & \mbox{if } \lambda_2\geq 0, \\ 
&\\
\left[\begin{array}{cc}0 & -\lambda_1\lambda_2 \\1 & \lambda_1+\lambda_2\end{array}\right], & \mbox{if } \lambda_2< 0.
\end{array}
\right.
\end{equation}
is a nonnegative matrix with eigenvalues $\{\lambda_1, \lambda_2\}$. Similarly, we can come up with different nonnegative matrices based on the conditions given in Lemma~\ref{LemNon}. These $2\times 2$ nonnegative matrices will play a decisive role in the solvability of the RNIEP and will be illustrated in the next section. 

Now we know how to define $2\times 2$ nonnegative matrices so that the constructed matrices possess a prescribed pair of eigenvalues. Next, an interesting question to ask is whether the specified eigenvalues can be applied to construct an $2\times 2$ symmetric nonnegative matrix. The answer can be provided by the following result. We omit the proof here since it is so similar to the discussion in Lemma~\ref{LemNon}.
\begin{lemma}\label{LemNon2}
$\{\lambda_1, \lambda_2\}$ are eigenvalues of a $2\times 2$ symmetric nonnegative matrix
$A = \left[\begin{array}{cc}a & b \\b & d\end{array}\right ]$
if and only if~\eqref{genrule} and the following  conditions,
\begin{align}\label{cond2}
\left.\begin{array}{rl}
\lambda_1 \geq a \geq \lambda_2,& 
\mbox{if }  \lambda_2 \geq 0, \\
\lambda_1+\lambda_2 \geq a \geq  0, &
\mbox{if }  \lambda_2 < 0,
%\\
%b = \sqrt{-a^2+(\lambda_1+\lambda_2)a - \lambda_1\lambda_2}\geq 0, 
\end{array}\right.
\end{align}
are satisfied. In particular, the entry $b$ is 
denoted by 
\[
b = \sqrt{-a^2+(\lambda_1+\lambda_2)a - \lambda_1\lambda_2}\geq 0.
\]
%Moreover, $b = \sqrt{-a^2+(\lambda_1+\lambda_2)a - \lambda_1\lambda_2}$.
\end{lemma}

From Lemma~\ref{LemNon2}, it can be seen that the matrix 
\begin{equation}\label{Sym}
A = \left\{\begin{array}{cc}
\left[\begin{array}{cc}\lambda_2 & 0 \\0 & \lambda_1\end{array}\right], & \mbox{if } \lambda_2\geq 0, \\ &\\
\left[\begin{array}{cc}0 & \sqrt{-\lambda_1\lambda_2} \\\sqrt{-\lambda_1\lambda_2} & \lambda_1+\lambda_2\end{array}\right], & \mbox{if } \lambda_2< 0.
\end{array}
\right.
\end{equation}
is a symmetric nonnegative  matrix with eigenvalues $\{\lambda_1, \lambda_2\}$.  In summary, the above examples can  
serve as building blocks to construct general $n\times n$ matrices with prescribed eigenvalues.

\section{Conquering procedure}
To facilitate our subsequent discussion, let $\rho(A)$ be the spectral radius of the nonnegative matrix $A$. It is known that $\rho(A)$ is an eigenvalue of  $A$, called the Perron eigenvalue,  and that there is a right eigenvector with nonnegative entries corresponding to the Perron eigenvalue. In this section we want to derive a sufficient condition for the set of $n$ real numbers $\lambda_1,\ldots,\lambda_n$ 
to be a possible set of eigenvalues in an $n\times n$ nonnegative matrix and then to come up with a numerical approach for constructing this nonnegative matrix.
To begin with, 
%the method applied are the application 
%of the theorem due to Nazari and Sherafat~\cite{Nazari2012}[Theorem2.1].  
%
we shall first present a useful result given in~\cite{Nazari2012}[Theorem 2.1] for combining the eigeninformation of two non-negative matrices. 
%This result has been applied in
%in~\cite{Fiedler1975} to provide a simple but more general proof of the Horn's theorem~\cite{Horn54}.
%\begin{theorem}
%Suppose $\{\lambda_k\}_{k=1}^n$ and $\{\beta_k\}_{k=1}^n$ are eigenvalues of an $n\times n$ nonnegative matrix $A$ and an $m\times m$ nonnegative matrix $B$, respectively.
%Let $u$ and $v$ be two unit eigenvectors corresponding to
%eigenvalues $\lambda_1$ and $\beta_1$, respectively.
%Then for any $\rho$, the matrix  
%\begin{equation}
%C = \left[\begin{array}{cc} A & \rho u v^\top \\
%\rho v u^\top & B\end{array}\right], 
%\end{equation}
%has eigenvalues $\lambda_2,\ldots,\lambda_n,\beta_2,\ldots,\beta_m,\gamma_1,\gamma_2$, where $\gamma_1$ and $\gamma_2$ are eigenvalues of the matrix  
%\begin{equation}
%\widehat{C} = \left[\begin{array}{cc} \alpha_1 & \rho  \\
%\rho & \beta_1\end{array}\right], 
%\end{equation}
%
%\end{theorem}

\begin{theorem}\label{nazari12}
Suppose $\{\lambda_k\}_{k=1}^n$ and $\{\beta_k\}_{k=1}^m$ are eigenvalues of an $n\times n$ nonnegative matrix $A$ and an $m\times m$ nonnegative matrix $B$, respectively, with $\lambda_1 \geq |\lambda_k|$ and $\beta_1 \geq |\beta_k|$ for all $k>1$. 
Let $\mathbf{v}$ be the unit eigenvector corresponding to the eigenvalue $\beta_1$.
If the matrix $A$ is of the form 
\begin{equation}
A = \left[\begin{array}{cc} A_1 & \mathbf{a} \\
\mathbf{b}^\top & \beta_1\end{array}\right], 
\end{equation}
where $A_1$ is an $(n-1)\times(n-1)$ matrix, and $\mathbf{a}, \mathbf{b}$ are two vectors in $\mathbb{R}^{n-1}$, 
then the set of the eigenvalues of the matrix 
%$(m+n-1)\times (m+n-1)$
%nonnegative matrix $C$ 
\begin{equation}\label{eq:c1}
C = \left[\begin{array}{cc} A_1 & \mathbf{a} \mathbf{v}^\top \\
\mathbf{v} \mathbf{b}^\top & B\end{array}\right], 
\end{equation}
is  $ \{\lambda_k\}_{k=1}^n\bigcup\{\beta_k\}_{k=2}^m$.

%$\sigma(C) = \{\lambda_k\}_{k=1}^n\bigcup\{\beta_k\}_{k=2}^n$.
\end{theorem}

Note that Nazari and Sherafat's result can also be obtained from an earlier perturbation result in~\cite[Theorem 11]{Smigoc2004} and  is quoted below. 
\begin{theorem}
Suppose $\{\lambda_k\}_{k=1}^n$ and $\{\beta_k\}_{k=1}^m$ are eigenvalues of an $n\times n$ nonnegative matrix $A$ and an $m\times m$ nonnegative matrix $B$, respectively, with $\lambda_1 \geq |\lambda_k|$ and $\beta_1 \geq |\beta_k|$ for all $k>1$. If $A$ has a diagonal entry $c$, then 
the list 
\begin{equation*}
(\lambda_1+\max\{\beta_1-c\},\lambda_2,\ldots,\lambda_n,\beta_2,\ldots,\beta_m)
\end{equation*}
is a set of eigenvalues of an $(n+m-1)\times (n+m-1)$ nonnegative matrix.
\end{theorem}

Also, the result of Theorem~\ref{nazari12} can be directly extended to the symmetric nonnegative matrices and has been discussed in~\cite[Theorem 8]{Laffey2007}.  
\begin{corollary}\label{nazari13}
Suppose $\{\lambda_k\}_{k=1}^n$ and $\{\beta_k\}_{k=1}^m$ are eigenvalues of a symmetric  $n\times n$ nonnegative matrix $A$ and a symmetric  $m\times m$ nonnegative matrix $B$, respectively, with $\lambda_1 \geq |\lambda_k|$ and $\beta_1 \geq |\beta_k|$ for all $k>1$. 
Let $\mathbf{v}$ be the unit eigenvector corresponding to the eigenvalue $\beta_1$.
If the matrix $A$ is of the form 
\begin{equation}
A = \left[\begin{array}{cc} A_1 & \mathbf{a} \\
\mathbf{a}^\top & \beta_1\end{array}\right], 
\end{equation}
where $A_1$ is an $(n-1)\times(n-1)$ matrix, and $\mathbf{a}, \mathbf{b}$ are two vectors in $\mathbb{R}^{n-1}$, 
then the set of the eigenvalues of the matrix 
%$(m+n-1)\times (m+n-1)$
%nonnegative matrix $C$ 
\begin{equation}\label{eq:c2}
C = \left[\begin{array}{cc} A_1 & \mathbf{a} \mathbf{v}^\top \\
\mathbf{v} \mathbf{a}^\top & B\end{array}\right], 
\end{equation}
is  
$ \{\lambda_k\}_{k=1}^n\bigcup\{\beta_k\}_{k=2}^m$.

%$\sigma(C) = \{\lambda_k\}_{k=1}^n\bigcup\{\beta_k\}_{k=2}^n$.

\end{corollary}

Based on Theorem~\ref{nazari12} or Corollary~\ref{nazari13}, we outline our ideas for the computation of a nonnegative matrix or symmetric nonnegative matrix, respectively, followed by a recursive algorithm. Our strategy is quite straightforward, but it offers an effective way to solve an RNIEP or SNIEP.
Here, we take the construction of a symmetric nonnegative matrix as an example. A similar approach can be applied to solve RNIEP with nonsymmetric cases and is demonstrated in section 4.
 Assume first that the set of 
eigenvalues $\{\lambda_1,\ldots,\lambda_n\}$ are
arranged in the order $\lambda_1\geq\ldots\geq\lambda_r\geq 0\geq\lambda_{r+1}\geq\ldots\geq \lambda_n$ and satisfy the condition 
\begin{equation}\label{condniep}
\lambda_1 +\lambda_{r+1}+\ldots+\lambda_n \geq 0.
\end{equation}
Note that condition~\eqref{condniep} is weaker than  %be regraded as an improvement to 
Suleimanova's result given in Theorem~\ref{Sulei}. To prove the existence of a nonnegative matrix of general dimensionality with eigenvalues $\{\lambda_i\}_{i=1}^n$, one can treat the iteration in terms of $2\times 2$  matrices step by step. For example, start with
a $2\times 2$ matrix $A$ as follows.

\emph{Case} 1. Suppose $\lambda_{2} \geq 0$ and  choose without loss of generality a $2\times 2$ matrix 
\begin{equation}
A = \left[\begin{array}{cc}\lambda_2 & 0 \\0 & \lambda_1\end{array}\right]. 
\end{equation}

\emph{Case} 2. Suppose $\lambda_{2} < 0$ and choose without loss of generality a $2\times 2$ matrix
\begin{equation}
A = \left[\begin{array}{cc} 0 & 
\sqrt{-\lambda_2\lambda_1} \\
\sqrt{-\lambda_2\lambda_1} & 
\lambda_1 + \lambda_2 \end{array}\right]. 
\end{equation}

Now let another eigenvalue $\lambda_3$ creep into the matrix $A$, i.e., obtain the matrix $C_1$ in%~\eqref{eq:c1},
~\eqref{eq:c2}, 
by suitably augmenting a new $2\times 2$ matrix $B_1$ with eigenvalues $\{\lambda_3, A(2,2)\}$ from Lemma
%~\ref{LemNon}.
 %or
 ~\ref{LemNon2}, 
 Again, two cases are required to be considered.

\emph{Case} 1. Suppose $\lambda_{3} \geq 0$ and  choose without loss of generality a $2\times 2$ matrix 
\begin{equation}
B_1 = \left[\begin{array}{cc}\lambda_3 & 0 \\0 & A(2,2)\end{array}\right]. 
\end{equation}
Here, $A(i,j)$ represents the $(i,j)$ entry of $A$.

\emph{Case} 2. Suppose $\lambda_{3} < 0$ and choose without loss of generality a $2\times 2$ matrix
\begin{equation}
B_1 = \left[\begin{array}{cc} 0 & 
\sqrt{-\lambda_3 A(2,2)} \\
\sqrt{-\lambda_3 A(2,2)} & 
\lambda_3 + A(2,2) 
\end{array}\right]. 
\end{equation}
Note that this augmentation is possible because by construction $A(2,2)\geq |\lambda_3|$.
It follows from %Theorem~\ref{nazari12} (or, 
Corollary~\ref{nazari13} that the matrix 
\begin{equation}\label{eq:c3}
C_1 = \left[\begin{array}{cc} A(1,1) &  A(1,2)\mathbf{v_1}^\top \\
\mathbf{v_1}A(1,2) & B_1\end{array}\right] 
\end{equation}
has a set of eigenvalues $\{\lambda_1, \lambda_2,
\lambda_3\}$, where $\mathbf{v_1}$ is the unit eigenvector corresponding to the Perron eigenvalue of the matrix $B_1$. Upon obtaining the matrix $C_1$ in~\eqref{eq:c3}, we shall replace entries of the original matrix $A$ with those of the new matrix $C_1$, i.e., redefine $A = C_1$, 
 and continue to augment  another $2\times 2$ matrix $B_2$ with eigenvalues $\{\lambda_4,A(3,3)\}$. 
 From 
% Theorem~\ref{nazari12} (or, 
% Corollary~\ref{nazari13}, respectively), 
Corollary~\ref{nazari13},
 the entry $A(3,3)$ in the lower right corner of the new matrix $A$ is required to be the Perron eigenvalue of the subsequent $2\times 2$ matrix $B_2$.  Since 
 by condition~\eqref{condniep}, 
\begin{equation*}
A(3,3)+\lambda_4\geq 0, \quad \mbox{if } \lambda_4  < 0, 
\end{equation*}
that is,  $A(3,3)\geq |\lambda_4|$, and  
$A(3,3)=\lambda_1\geq \lambda_4$, if $\lambda_4  \geq 0 $, it follows from %Lemma~\ref{LemNon} (or, 
Lemma~\ref{LemNon2}
%, respectively) 
that
there is a nonnegative matrix $B$ with eigenvalues 
$\{\lambda_4, A(3,3)\}$. Therefore another new matrix $C_2$ can be defined by
\begin{equation*}
C_2 = \left[\begin{array}{cc} A(1:2,1:2) &  A(1:2,3)\mathbf{v}_2^\top \\
\mathbf{v_2} A(1:2,3)^\top& B_2\end{array}\right], 
\end{equation*}
where $\mathbf{v_2}$ is the unit eigenvector corresponding to the Perron eigenvalue $A(3,3)$ of the matrix $B_2$ and $A(i:j,k)$ represents a column vector 
defined by $A(i:j,k) =[A(i,k),\ldots,A(j:k)]^\top$.
%, and $A(k,i:j)$ represents a row vector defined by
%$A(k,i:j) = [A(k,i),\ldots,A(k,j)]$. 
We then redefine the matrix $A$ by $A = C_2$ with eigenvalues $\{\lambda_1,\lambda_2,\lambda_3,\lambda_4\}$.
Again, by condition~\eqref{condniep}, the entry $A(4,4)$ in the lower right corner of $A$ can be served as the Perron eigenvalue of a nonnegative matrix $B_3$ with eigenvalues $\{\lambda_5,A(4,4)\}$.
We then continue the above process for the next category, and finally obtain 
a constructive way for the solution of an $n\times n$ nonnegative matrix $A$ with eigenvalues $\{\lambda_1,\ldots,\lambda_n\}$. 

The above recursive process 
for obtaining a nonnegative matrix with the desired spectrum $\{\lambda_1,\ldots,\lambda_n\}$ can be conveniently demonstrated
in MATLAB expressions as in Algorithm~\ref{alto}. 
More specifically, the matrix obtained by Algorithm~\ref{alto} is explicitly a symmetric matrix, but also implies the capacity of solving RNIEP.
%which therefore implies that the above approach can be limited to problems of symmetric cases.
%
%In Algorithm~\ref{alto}, we select a sequence of $2\times 2$ matrices so that the constructed nonnegative matrix with desired spectrum is symmetric as an example.  We can also apply Lemma~\ref{LemNon} for constructing other types of nonnegative matrices.  
In other words, it is quite intriguing that different approaches using different sets of $2\times 2$ matrices end up with different kinds of nonnegative matrices with prescribed eigenvalues. We apply the following example to demonstrate this property
more fully.

\begin{example}
Given eigenvalues $\{2,\frac{1}{2},-1\}$, this example follows from the approach given in Algorithm~\ref{alto}. We might select the initial matrix $A$ as 
\begin{equation*}
A = \left[\begin{array}{cc}\frac{1}{2} & 0 \\0 & 2\end{array}\right]. 
\end{equation*}
and two kinds of matrix $B$ as 
\begin{equation*}
B = \left[\begin{array}{cc}0 & \sqrt{2} \\\sqrt{2}  & 1\end{array}\right] \mbox{ or }
 B = \left[\begin{array}{cc}0 & 2 \\1 & 1\end{array}\right]. 
\end{equation*}
This implies that the matrix $C$ obtained from the combination of matrices $A$ and $B$ is written as 
\begin{equation*}
C = \left[\begin{array}{ccc}\frac{1}{2}& 0& 0 \\
0& 0& \sqrt{2} \\
0& \sqrt{2}& 1 \end{array}\right] \mbox{ or }
C = \left[\begin{array}{ccc}\frac{1}{2}& 0& 0 \\
0& 0& 2 \\
0& 1& 1 \end{array}\right] 
\end{equation*}
 with eigenvalues $\{2,\frac{1}{2},-1\}$. One is symmetric nonnegative matrix and the other is just nonnegative matrix.
\end{example}
\begin{algorithm}
\caption{The RNIEP/SNIEP: \hfill $[A] = \textsc{RNIEP/SNIEP}(\Lambda)$} \label{alto}
\begin{algorithmic}
\STATE{ {\bf Given} $\Lambda = \rm{diag}(\lambda_1,\ldots,\lambda_n)$, where $\lambda_1\geq \lambda_2\geq\ldots\geq\lambda_n$,  
}

%
%\vskip -.2in%\Indm
% }
\RETURN{a $n\times n$ symmetric nonnegative matrix $A$ 
that is isospectral to $\Lambda$.
 }
\STATE  \% Set up an initial $2\times 2$ matrix
 \IF {$\lambda_2 \geq 0$}  \STATE 
\STATE 
$A \leftarrow  \left[\begin{array}{cc}\lambda_2 & 0 \\0 & \lambda_1
\end{array}\right]$;
\ELSE
\STATE $A \leftarrow  \left[\begin{array}{cc}0 & \sqrt{-\lambda_2 \lambda_1} \\
\sqrt{-\lambda_2 \lambda_1} & \lambda_1+\lambda_2
\end{array}\right]$;
\ENDIF

\STATE \% Conquering procedure
\FOR{$i = 3 \ldots n$}
      \IF {$\lambda_i \geq 0$}
 \STATE  $B \leftarrow \left[\begin{array}{cc}\lambda_i & 0 \\0 & A(i-1,i-1)\end{array}\right]$;    
       \ELSE
 \STATE  $B \leftarrow \left[\begin{array}{cc}0 & 
                  \sqrt{-\lambda_i A(i-1,i-1)} \\
             \sqrt{-\lambda_i A(i-1,i-1)} & A(i-1,i-1)+           
             \lambda_i\end{array}\right]$;               
        \ENDIF
   \STATE \% Compute the Perron eigenvector of $B$.     
  \STATE{$[v] = \textsc{PerronEigvector}(B)$}
   \STATE{\% Apply Theorem~\ref{nazari12}/Corollary~\ref{nazari13}}      
        
\STATE { 
$\textsc{temp} \leftarrow  A(1:i-2,i-1)*v^\top$;}
\STATE { 
 $A \leftarrow  \left[\begin{array}{cc}A(1:i-2,1:i-2) & \textsc{temp} \\
\textsc{temp}^\top & B
\end{array}\right]$,
}

\ENDFOR
\end{algorithmic}
\end{algorithm}

%
%
%
%
%
%
%
%Let $B$ be a matrix with eigenvalues $\{\lambda_i, A(i-1,i-1)\}$. Here $A(i-1,i-1)$ represents the $(i-1,i-1)$ entry of the $(i-1)\times (i-1)$ matrix $A$. 
%
%%with a set of eigenvalues $\lambda_1,\ldots,\lambda_{i-1}$.  
%
%
%Now we want to let another eigenvalue $\lambda_i$  creep into the matrix $B$. This step can be accomplished by choosing a suitable matrix $B$ from 
%the results given in Lemma~\ref{LemNon} or~\ref{LemNon2}. For example,  
%
%
%taking the following two $2\times 2$ matrices into consideration. 
%
%
%
%\emph{Case} 1. Suppose $\lambda_{i} \geq 0$ and  consider without loss of generality a $2\times 2$ matrix 
%\begin{equation}
%B = \left[\begin{array}{cc}\lambda_i & 0 \\0 & A(i-1,i-1)\end{array}\right]. 
%\end{equation}
%It then follows from Theorem~\ref{nazari12} that the matrix 
%\begin{equation}\label{c1}
%C = \left[\begin{array}{cc} \lambda_i &  0 \\
%0 & B\end{array}\right] 
%\end{equation}
%has a set of eigenvalues $\{\lambda_k\}_{k=1}^i$.
%

Note that in our algorithm, we 
break down the construction of the desired matrix $A$ to a sequence of submatrices of size $2$
and then combine these submatrices together to give a nonnegative solution to the original problem. 
Based on this conquering procedure, we then have the following sufficient condition for the construction of a nonnegative matrix and its proof can be directly observed from the above discussion.

%
%We then implement the result for constructing 
%nonnegative matrices. 
%
\begin{theorem}\label{SuleiMin}
Let $\lambda_1 \geq \lambda_2\geq \ldots \geq \lambda_n$ be real numbers and let $r$ be the greatest number with $\lambda_r \geq 0$. If 
the condition
\begin{equation}\label{Min}
\lambda_1 +\lambda_{r+1}+\ldots+\lambda_n \geq 0
\end{equation}
is satisfied, then there exists an $n\times n$ nonnegative matrix with $\sigma = \{\lambda_1,\ldots,\lambda_n\}$ as its spectrum. Indeed, the nonnegative matrix can also be chosen to be a symmetric matrix.
% {\bf or even a stochastic matrix}
\end{theorem}

Note that the sufficient condition in Theorem~\ref{SuleiMin} does not require the negativity of the remaining $n-1$ eigenvalues and is somewhat weaker than that given by Suleimanova.  This simplified condition is also shown in~\cite{Perfect1952} for the cases $n=2$ and $3$ by a geometric point of view. 

\section{Numerical experiments}
In this section, we demonstrate 
by numerical examples how Algorithm~\ref{alto} can be applied to construct the solution of RNIEP
 (or, SNIEP) and stochastic matrices associated with some particular spectrum. 
 
%In the first example, we test our algorithm 
%so that we can test our algorithm again problems of relatively large sizes.

\begin{example}
To illustrate the feasibility of our approach again problems of relatively large size, we being with a set of eigenvalues of size larger than $5$. To demonstrate the robustness of our approach, the test data is generated from a uniform distribution over the interval $[-10,0]$, say $\{21.3323,    5.0851,   3.0635,\\   -5.1077,   -7.9483,   -8.1763\}$. It can be easily seen that this set of eigenvalues satisfies condition~\eqref{Min}. Reported below is one typical result in our experiment.
\begin{equation*}
A = \left[\begin{array}{cccccc}
    5.0851 &        0&         0&         0&         0&         0\\
         0   & 3.0635     &    0        & 0       &  0  &       0\\
         0    &     0       &  0 &   5.9856    & 6.0286 &   6.0653\\
         0     &    0    &5.9856&         0   & 8.0054   & 8.0542\\
         0      &   0   & 6.0286 &   8.0054  &       0  &  8.2261\\
         0       &  0&    6.0653  &  8.0542&    8.2261&    0.1000
\end{array}\right]
\end{equation*}
We note that the original algorithm considers a symmetric nonnegative matrix as the target. As is expected, the output result can be a general nonnegative matrix with the same spectrum. This result can be obtained by choosing the updated matrices as\begin{equation*}
B = \left[\begin{array}{cc}0 &  -A(i-1,i-1)\lambda_i \\1 &  A(i-1,i-1)+\lambda_i\end{array}\right], 
\quad \mbox{if } \lambda_i< 0,
\end{equation*}
and the reported result is 
%
%We want to apply Chu's method to generate an artificial example....
%
%We being with a small artificial example demonstrating how to apply the algorithm. Given eigenvalues $a$
\begin{equation*}
A = \left[\begin{array}{cccccc}
    5.0851&         0   &      0    &     0     &    0\\
         0 &   3.0635    &     0       &  0     &    0\\
         0   &      0   &      0  &108.1067  & 13.6012\\
         0  &       0 &   0.9922 &        0 & 128.9580\\
         0   &      0   & 0.1248  &  1.0000 &   8.2763
\end{array}\right].
\end{equation*}

\end{example}

 \begin{example}
In this example, the well-known Suleimanova's result 
in Theorem~\ref{Sulei} is given to test our approach. To begin with, we randomly generate a set of negative eigenvalues, for example
$\{5.4701, 2.9632, 7.4469, 1.8896\}$ 
from the uniform distribution on the interval $[-10,0]$. 
We might select without loss of generality a positive eigenvalue $17.8698$ so that the Suleimanova's condition is satisfied. Using this spectrum, a desired nonnegative matrix can then be computed by applying Algorithm 1 as follows: 
\begin{equation*}
A =\left[\begin{array}{ccccc}
         0  &  2.2982 &   2.9032 &   3.1562 &   3.1773\\
      2.2982 &    0&    3.7431&    4.0693&    4.0966\\
    2.9032 &   3.7431&         0&    5.9468&    5.9866\\
    3.1562&    4.0693&    5.9468&         0&    7.4967\\
    3.1773&    4.0966&    5.9866&    7.4967&    0.1000
    \end{array}\right]
\end{equation*}

 \end{example}

 \begin{example}
 In this example, we illustrate the application of our approach to construct a stochastic matrix with a prescribed spectrum. This is the so-called inverse stochastic eigenvalue problem. Note that the inverse eigenvalue problem for nonnegative matrices is practically equivalent to that for stochastic matrices.
For example, if  $\{\lambda_1,\ldots,\lambda_n\}$ 
with $\lambda_1 = \max_{1\leq i\leq n}|\lambda_i|$
is the set of eigenvalues of an $n\times n$ nonnegative matrix, then it is known that $\{1,\lambda_2/\lambda_1,\ldots,\lambda_n/\lambda_1\}$ is the spectrum of a $n\times n$
row stochastic matrix~\cite{Wuwen1996}[Lemma 5.3.2]. Our approach is first to construct nonnegative matrix with the given spectrum and then transform the nonnegative matrix to a stochastic matrix based on the following theorem~\cite{Minc1988}.
\begin{theorem}\label{minc}
Suppose $A$ is a nonnegative matrix with a positive maximal eigenvalue $\rho(A)$ and a positive eigenvector $\mathbf{x}=[x_{i}]$ such that $A\mathbf{x} = \rho(A)\mathbf{x}$. Let $D = [d_{ij}]$ be a diagonal matrix with diagonal entries defined by $d_{ii} = {x}_i$. Then $\dfrac{1}{\rho(A)}D^{-1}AD$ is a stochastic matrix.
\end{theorem}

The example experimented here is taken from~\cite{Chu1998}. It is to find a stochastic matrix with 
(presumedly randomly generated)
eigenvalues $\{1.0000, -0.2608,\\0.5046,0.6438,-0.4483\}$. To facilitate our illustration, assume the eigenvalues 
have been arranged in the decreasing order such that 
$\lambda_1 = 1.0000,\lambda_2=0.6438,
\lambda_3 = 0.5046, \lambda_4 = -0.2608$ and $\lambda_5 = -0.4483
$. 
Note that in order to apply Theorem~\ref{minc}, the constructed nonnegative matrix should have a positive eigenvector corresponding to a positive maximal eigenvalue. For this purpose, we have to fine-tune Algorithm~\ref{alto}, while including positive eigenvalues into a matrix. 
This adjustment is a simply application of Lemma~\ref{LemNon} by 
computing $\textsc{Neg} = \lambda_4 + \lambda_5$, choosing the initial value $A$
as
\begin{equation*}
A = \left[\begin{array}{cc}
a
 & b \\c & 
d
\end{array}\right],
\end{equation*}
where 
$a = \lambda_2+\frac{\lambda_1+\textsc{Neg}}{2}$,
$d =  \frac{\lambda_1-\textsc{Neg}}{2}$
and $b=c = \sqrt{ad-\lambda_1\lambda_2}$, and selecting the subsequent matrix $B$ as
\begin{equation*}
B = \left[\begin{array}{cc}
e
 & f \\g & 
h
\end{array}\right]
\end{equation*}
where 
$e = \lambda_3+\frac{A(2,2)+\textsc{Neg}}{4}$,
$h =  \frac{3A(2,2)-\textsc{Neg}}{4}$
and $g=f = \sqrt{eh-\lambda_3A(2,2)}$.  
It is true that by Lemma~\ref{LemNon}, we have  many different choices for selecting matrices $A$ and $B$. Our methodology used here is to construct an irreducible nonnegative matrix in the end. It then follows from the Perron-Frobenius theorem~\cite{Gantmacher1959} that for this nonnegative matrix, there is a positive eigenvalue associated with an eigenvector which can be chosen to be entry-wise positive.  It follows that an example of a stochastic matrix with the desired spectrum is 
\begin{equation}
A = \left[\begin{array}{ccccc}
    0.7893  &  0.0219&    0.0456  &  0.0638 &   0.0794\\
    0.1454  & 0.5410 &   0.0758   & 0.1060  &  0.1318\\
    0.1454  &0.0364   &      0 &   0.3647 &   0.4535\\
    0.1454  &0.0364   & 0.2608 &        0   & 0.5574\\
    0.1454  &0.0364    & 0.2608   & 0.4483 &   0.1091
\end{array}\right]
\end{equation}

 In~\cite{Chu1998}, this example is further restricted to a structured stochastic matrix with the zero pattern 
 given by the zeros of the following matrix:
 \begin{equation*}
 Z = \left[\begin{array}{ccccc}1 & 1 & 0 & 0 & 1 \\1 & 1 & 1 & 0 & 0 \\0 &1 & 1 & 1 & 0 \\0 & 0 & 1 & 1 & 1 \\1 & 0 & 0 & 1 & 1\end{array}\right]
 \end{equation*}
 Our algorithm as it stands can not solve this problem directly though we might be able to select a sequence of particular matrices so that the obtained nonnegative matrix has structure corresponding or similar to the matrix $Z$. However, unlike the methods proposed in~\cite{Chu1998,Orsi2006}, our methodology is computed by simply combing a sequence of $2\times 2$ matrices, that is, the computed result can preserve the desired spectrum with high precision.

 \end{example}
 
% 
%To simplify my consideration, I assume that the given matrix is symmetric so that the desired eigenvalues are real and are arranged in the descending order. 
%
%In terms of Theorem 2.1, the key point
%is to construct $A(i-1,i-1)$ as large as possible (see lines 44 and 57 in the attached code.) 
%These two lines are unique in order to get the largest entry in $A(i-1,i-1)$. 
%But, this strategy has a limit and is not available anymore if $A(i-1,i-1) + dD(i) < 0$. 
%This is due to the fact in line 57 that the $A(i-1,i-1)$ is always equal to $dD(1)$, 
%instead of $dD(1)+dD(2)+\ldots+dD(k)$, if $dD(k) \geq 0$. 
%Thus, line 44 will reduce the entry in $A(i-1,i-1)$ and cause that $A(i-1,i-1) + dD(i) < 0$, that is,
%further expansion is not available anymore  if $A(i-1,i-1) + dD(i) <0$. 
% 

\section{Conclusion}
Determining the necessary and sufficient conditions of solving inverse eigenvalue problems for nonnegative matrices or symmetric nonnegative matrices is very challenging and the conditions for matrices of larger size remain unknown. The main thrust of this paper is to present a numerical procedure for constructing a nonnegative matrix or a symmetric nonnegative matrix provided that the desired spectrum is given. With slight modification, our method can solve inverse eigenvalue problems for stochastic matrices as well.  
The crux of our algorithm is the employment of 
Nazari and Sherafat's result~\cite{Nazari2012}. At each step, we look for a sequence of $2\times 2$ matrices with the desired eigenvalues and a desired structure such as symmetry and combine them together for solving the RNIEP or SRIEP. We then propose, based on our procedure, a weaker sufficient condition for solving RNIEP than Suleimanova's result and a condition for solving SNIEP. 

From the existing structured inverse eigenvalue problems, this paper describes only a numerical procedure for symmetric nonnegative matrices
and stochastic matrices. However this procedure might serve as a possible computational tool for inverse eigenvalue problems involving many other types of structured nonnegative matrices such as Toeplitz, Hankel, and others. In addition, we also propose, based on our procedure, a weaker sufficient condition for solving RNIEP than Suleimanova's result and a condition for solving SNIEP. The application of our conquering strategy for structured inverse eigenvalue problems is a subject worthy of further investigation.

\section*{Acknowledgement}

The authors wish to thank Professor Moody Chu for valuable suggestions on the manuscript. This research work is partially supported by the National Science Council and the National Center for Theoretical Sciences in Taiwan.

\bibliographystyle{elsarticle-num}
%\bibliography{qiep}

\end{document}